\documentclass[a4paper,11pt]{amsart}

\usepackage[usenames,dvipsnames]{pstricks}
\usepackage{epsfig}
\usepackage{pst-grad} 
\usepackage{pst-plot} 
\usepackage[utf8]{inputenc}

\usepackage{amssymb}
\usepackage{mhequ}

\usepackage{cite}

\DeclareFontFamily{OT1}{rsfs}{}
\DeclareFontShape{OT1}{rsfs}{m}{n}{ <-7> rsfs5 <7-10> rsfs7 <10->
rsfs10}{} \DeclareMathAlphabet{\mathscr}{OT1}{rsfs}{m}{n}

\newcommand{\bel}[1]{\begin{equation}\label{#1}}
\newcommand{\beal}[1]{\begin{eqnarray}\label{#1}}
\newcommand{\beadl}[1]{\begin{deqarr}\label{#1}}
\newcommand{\eeadl}[1]{\arrlabel{#1}\end{deqarr}}
\newcommand{\eeal}[1]{\label{#1}\end{eqnarray}}
\newcommand{\eead}[1]{\end{deqarr}}
\newcommand{\eea}{\end{eqnarray}}
\newcommand{\eeaa}{\end{eqnarray*}}
\newcommand{\beaa}{\begin{eqnarray*}}
\newcommand{\be}{\begin{equation}}
\newcommand{\ee}{\end{equation}}

\DeclareFontFamily{OT1}{rsfs}{}
\DeclareFontShape{OT1}{rsfs}{m}{n}{ <-7> rsfs5 <7-10> rsfs7 <10->
rsfs10}{} \DeclareMathAlphabet{\mycal}{OT1}{rsfs}{m}{n}

\newcounter{mnotecount}[section]

\newcommand{\rmnote}[1]{}

\def\mysavedown#1{\edef\mysubs{\mysubs#1}}
\def\mysaveup#1{\edef\mysups{\mysups#1}}
\def\mydown#1{{\mytensor}_{\vphantom{\mysubs}#1}}
\def\myup#1{{\mytensor}^{\vphantom{\mysups}#1}}
\def\tensor#1#2{
  #1
  \def\mytensor{\vphantom{#1}}
  \def\mysubs{\relax}
  \def\mysups{\relax}
  \let\down=\mysavedown
  \let\up=\mysaveup
  #2
  \let\down=\mydown
  \let\up=\myup
  #2
  }

\renewcommand{\sharp}{\#}

\newcommand{\R}{\mathbb R}

\renewcommand{\to}{\rightarrow}

\renewcommand{\epsilon}{\varepsilon}
\renewcommand{\hat}{\widehat}

\def\crn#1#2{{\vcenter{\vbox{
        \hbox{\kern#2pt \vrule width.#2pt height#1pt
           }
          \hrule height.#2pt}}}}

\renewcommand{\hbar}{{\overline h}}

\newcommand{\pre}[2]{{{\vphantom{#2}}^{#1}}\kern-.2ex{#2}}

\DeclareMathOperator{\dist}{dist}

\renewcommand{\S}{\mathbb S}

\theoremstyle{plain}
\newtheorem{theorem}{\sc Theorem}[section]
\newtheorem{lemma}[theorem] {\sc Lemma}
\newtheorem{proposition}[theorem]{\sc Proposition}

\theoremstyle{definition}

\newtheorem{remark}[theorem]{\sc  Remark\rm}

\numberwithin{equation}{section}

\date{\today}

\begin{document}

\author[W. Ao]{Weiwei Ao}

\address{Weiwei Ao
\hfill\break\indent
Wuhan University
\hfill\break\indent
Department of Mathematics and Statistics, Wuhan, 430072, PR China}
\email{wwao@whu.edu.cn}

\author[M.d.M. Gonz\'alez]{Mar\'ia del Mar Gonz\'alez}

\address{Mar\'ia del Mar Gonz\'alez
\hfill\break\indent
Universidad Aut\'onoma de Madrid
\hfill\break\indent
Departamento de Matem\'aticas, 28049 Madrid, Spain}
\email{mariamar.gonzalezn@uam.es}

\author[Y. Sire]{Yannick Sire}

\address{Yannick Sire
\hfill\break\indent
Johns Hopkins University
\hfill\break\indent
Department of Mathematics, Baltimore, MD 21218, USA}
\email{sire@math.jhu.edu}

\title{Boundary connected sum of Escobar manifolds}

\begin{abstract}
Let $(X_1, \bar g_1)$ and $(X_2, \bar g_2)$ be two compact Riemannian manifolds with boundary $(M_1,g_1)$ and $(M_2,g_2)$ respectively. The Escobar problem consists in prescribing a conformal metric on a compact manifold with boundary with zero scalar curvature in the interior and constant mean curvature of the boundary.  The present work is the construction of a connected sum $X=X_1 \sharp X_2$ by excising half ball near points
on the boundary.  The resulting metric on $X$ has zero scalar curvature and a CMC boundary. We fully exploit the nonlocal aspect of the problem and use new tools developed in recent years to handle such kinds of issues. Our problem is of course a very well-known problem in geometric analysis and that is why we consider it but the results in the present paper can be extended to other more analytical problems involving connected sums of constant fractional curvatures.
\end{abstract}

\maketitle

\tableofcontents

\section{Introduction}\label{section:intro}
Let $(X_i,\bar g_i)$, $i=1,2$, be two $(n+1)$-dimensional smooth compact Riemannian manifolds with boundaries $M_i$, $i=1,2$, for some $n\geq 2$, and set $g_i=\bar g_i|_{M_i}$. We assume that  $X_i$ are scalar-flat in the interior and have constant mean curvature $H_0$ on the boundary $M_i$.

We are interested in constructing connected sums on the boundary, producing a new scalar-flat manifold with CMC of the boundary. This is related to the so-called Escobar problem (see \cite{escobar}). Indeed, let $u$ be a (positive) solution on $X^{n+1}$ of the problem
\[
\left\{
\begin{array}{rllllll}
-    \, \Delta_{X}  u +\frac{n-1}{4n} R_{X} u & =& 0 \qquad \mbox{in} \qquad X,\\[3mm]
 \,  \partial_\nu u + \frac{n-1}{2}H_M \, u & = &  \frac{n-1}{2}H_0 \, u^{\frac{n+1}{n-1}} \qquad \mbox{on} \qquad \partial X,
\end{array}
\right.
\]
where $R_X$ is the scalar curvature of $X$, $H_M$ the mean curvature of $M=\partial X$, $\nu$ is the outer normal with respect to the metric $g$ (on $\partial X$) and $H_0$ is a constant depending only on the conformal structure. Therefore, the new metric
$$\bar g'= u^{\frac{4}{n-1}} \bar g$$
has zero curvature and the boundary has constant mean curvature with respect to the metric $\bar g'$.  Escobar in his seminal paper \cite{escobar} solved the previous boundary problem in most of the cases (see also \cite{Marques:existence-results,Brendle-Chen,Mayer-Ndiaye} for later developments on the problem).

The aim of the present work is to provide a different construction based on connected sums. While connected sums in the interior have been obtained in \cite{MPU,joyce}, in this paper, we construct connected sums on the boundary. To this end, we fully use the analogy of the Escobar problem with the Dirichlet-to-Neumann operator approach as pointed out in \cite{GG,CG} (see also the survey \cite{Gonzalez:survey}). Following the work of Graham and Zworski \cite{GZ} and many others (see, for instance, \cite{Guillarmou:meromorphic,Case-Chang} and the references therein), one can see that the Dirichlet-to-Neumann operator associated to the Escobar problem is a conformally covariant pseudo-differential operator of order $1/2$.

We prove the following theorem:
\begin{theorem}\label{main}
Let $(X_i,\bar g_i)$, $i=1,2$ be two smooth compact Riemannian manifolds with boundary $(M_i,g_i)$, $g_i=\bar g_i|_{M_i}$, $i=1,2$. Assume that the manifolds $X_i$ are non degenerate in the sense below, scalar-flat and with CMC boundary. Then there exists a connected sum $X_1\sharp X_2$ obtained by excising an half-ball around a boundary point which is scalar-flat and has CMC boundary.
\end{theorem}

The previous statement is well-known if one does connected sums in the interior as in the works of Joyce \cite{joyce} and Mazzeo, Pollack and Uhlenbeck \cite{MPU}. Our contribution is the development of new tools to handle connected sums on the boundary. In this case the problem becomes non-local since one has to deal with the Dirichlet-to-Neumann operator. While Escobar's problem is an old one, and gluing methods for constant positive scalar curvature are well known by now, when dealing with boundary gluing problems  non-degeneracy is  the main open question.

Due to the recent progress on non-local operators in conformal geometry, in particular \cite{ACDFGW}, several major issues such as non-degeneracy can be by now resolved. Our point of view is to consider the boundary operator directly as a lower order perturbation in H\"ormander classes of the half-Laplacian.

Other gluing problems for non-local equations have been considered in \cite{Existence-weak-solutions,Gluing-isolated,ACDFGW} with the objective of constructing singular metrics with prescribed non-local curvature. Another problem where non-degeneracy was a crucial point was in \cite{Supercritical}.

\section{Formulation of the problem}

As in the standard case (connected sums for Yamabe metrics), the idea is to use a perturbative argument, based on the non-degeneracy of each piece of the connected sum. In our context, this has to be done at a pseudo-differential level.

Let $(X,\bar g)$ be an $(n+1)-$dimensional smooth compact Riemannian manifold with boundary $(M=\partial X, g=\bar g|_{M})$ and consider the following boundary problem
\begin{equation}\label{Dirichletfv}
\left\{
\begin{array}{rllllll}
L_{\bar g}u\,\,\,:=\,\,-   \Delta_{\bar g}  u + \frac{n-1}{4n}R_{\bar g} u & =& 0 \qquad \mbox{in} \qquad X,\\[3mm]
 u & = &  f \qquad \mbox{on} \qquad M,
\end{array}
\right.
\end{equation}
where $R_{\bar g}$ is the scalar curvature of $(X,\bar g)$.
It has been proved in \cite{GG} and \cite{CG} (see the case $2$, $\gamma=1/2$, of Theorem 5.1) that a solution $u$ of \eqref{Dirichletfv}
exists and is unique.   This allows to define the Dirichlet-to-Neumann type operator by
\begin{equation*}\label{half-Laplacian}f\mapsto P_g f=\partial_\nu u|_{\partial X} +  \tfrac{n-1}{2}H_g \, f\end{equation*}
where $H_g$ is the mean curvature of the boundary $M$ with respect to the metric $g$.
 We also denote
$$Q_g:=P_g(1).$$
\begin{remark}
Notice that $Q_g$ is different from $\frac{n-1}{2}H_g$ when  $R_{\bar g}$ is not zero.
\end{remark}

The operator pair $(L_{\bar g},P_g)$ is conformally convariant, indeed, under the change of metric $\bar g_u=u^{\frac{4}{n-1}}\bar g$, and consequently, $g_f=f^{\frac{4}{n-1}}g$ for $f=u|_{M}$, we have that
\begin{equation}\label{transformation-law}
\left\{
\begin{array}{rllllll}
L_{\bar g_u} (\cdot)&=&u^{-\frac{n+1}{n-1}} L_{\bar g}(u\, \cdot)
 \qquad \mbox{in} \qquad X,\\[3mm]
P_{g_f}(\cdot)&=&f^{-\frac{n+1}{n-1}}P_{g}(f \,\cdot) \qquad \mbox{on} \qquad M.
\end{array}
\right.
\end{equation}
By evaluating at the function identically one, one obtains the equation pair:
\begin{equation*}
\left\{
\begin{array}{rllllll}
L_{\bar g} u&=& \frac{n-1}{4n}R_{\bar g}u^{\frac{n+1}{n-1}}  \qquad \mbox{in} \qquad X,\\[3mm]
P_{g}f&=& Q_{g_f}f^{\frac{n+1}{n-1}} \qquad \mbox{on} \qquad M,
\end{array}
\right.
\end{equation*}
where $u|_M=f.$

As a consequence, the problem under consideration is written as
\begin{equation}\label{problem}
\left\{
\begin{array}{rllllll}
L_{\bar g} u&=& 0  \qquad \mbox{in} \qquad X,\\[3mm]
P_{g}f&=& cf^{\frac{n+1}{n-1}} \qquad \mbox{on} \qquad M,
\end{array}
\right.
\end{equation}
for a constant $c$. In order to simplify our notation we will simply write
\begin{equation}\label{theequation}
\mathcal Q_g(f):=f^{-\frac{n+1}{n-1}}P_{g}f=c
\end{equation}
for problem \eqref{problem} and the extension equation will be implicitly understood.

In this paper we will only consider the case $c$ positive. The Yamabe problem in the non-positive case is much easier to handle and follows directly from the variational methods in \cite{escobar}.


By construction, the operator $P_g$ is a pseudo-differential operator of order $1$ (recall the extension formulation for the half-Laplacian \cite{Caffarelli-Silvestre,CG}). Furthermore,
$$P_g=(-\Delta_g)^{1/2}+\Psi_g,$$
where $\Psi_g$ is a pseudo-differential operator of lower order. Then the analysis for \eqref{problem} is inspired by equation
\begin{equation*}
(-\Delta_g)^{1/2} f= c f^{\frac{n+1}{n-1}}.
\end{equation*}

We could think now of the problem
\begin{equation*}
(-\Delta_g)^{\gamma} f= c f^{\frac{n+2\gamma}{n-2\gamma}},\quad \gamma\in(0,1).
\end{equation*}
The gluing construction for other powers $\gamma\in(0,1)$ in order to have connected sums of constant non-local $\gamma$-curvature  should still hold, but it is more  delicate since when building an approximate metric $\bar g_\epsilon$  in Section \ref{section:approximate}, we move outside of the conformal class of $\bar g$. The technicalities involved should be the same ones as in Section 6.8 of \cite{Gonzalez:survey}.\\

We conclude this section by noting that a key operator is the linearization of $P_g$ around a constant $f\equiv 1$ defined on $M$. Since $P_g$ is linear we see that
\begin{equation}\label{linearized}
\mathcal L_g v:=D\mathcal Q_g (1)v=\left.\frac{d}{dt}\right|_{t=0}\mathcal Q_g(1+tv)=P_gv-\frac{n+1}{n-1}Q_gv .
\end{equation}

\section{Construction of the approximate metric}\label{section:approximate}

To construct the approximate solution, we follow the approach of Mazzeo, Pollack and Uhlenbeck \cite{MPU} and, in particular, the one of Mazzeo and Pacard \cite{Mazzeo-Pacard}. There the gluing procedure to obtain $X_1\sharp X_2$ is along an half-ring around a point at infinity. Here we consider the same type of construction at a point on the boundary.

Let $(X_i,\bar g_i)$, $i=1,2$ be two manifolds with boundary  $(M_i=\partial X_i,g_i=\bar g _i|_{M_i})$, $i=1,2$ as in the hypothesis of Theorem \ref{main}.
Consider $p_1 \in  M_1=\partial X_1$ and let $\left \{ y_\ell\right \}$ be a system of normal coordinates on $M_1$ centered at $p_1$. We set Fermi coordinates  $z=(x,y)\in X_1$  and the annulus  $$A_{\epsilon}=\{\epsilon\delta<|z|<\delta\},$$
where  $|z|^2=x^2+|y|^2$. We will take $\delta=\delta(\epsilon)$ satisfying $\delta\to 0$ but  this will not make a difference in the following as long as  $\epsilon^{1/2}<<\delta$ as $\epsilon\to 0$.

In Fermi coordinates, the metric is written as
$$\bar g_{1}=(dx^2+dy^2)+O(\delta^2).$$
Thanks to the dilation
$$R_\delta:A_1\ni z\longrightarrow \delta z\in A_\delta,$$
one can work on $A_1$ (the dependence of $\delta$ should still be taken into account, though). The metric on $A_1$ becomes
$$
\delta^2(dx^2+dy^2)+O(\delta^4).
$$
Multiplying by $\delta^{-2}$ (the metric stays scalar-flat with CMC on the boundary), one can work with the new metric
$$
\bar g_{1,\delta}=(dx^2+dy^2)+O(\delta^2).
$$

Then one reproduces the same construction on $X_2$ in an annulus $A'_1=\{\epsilon<|z'|<1\}$, parameterized with Fermi coordinates $z'=(x',y')$ near a point $p_2\in M_2$. Thanks to an inversion $z'=I(z)=\frac{z}{|z|^2}\epsilon $, one can identify both annulus  $A_1$ and $A_1'$, and we can define a smooth
 manifold with boundary
\begin{equation*}
X :=X_1\sharp X_2.
\end{equation*}
Denote also $M=\partial X$. Note that, often, we will just refer to the manifold $M$, even though the manifold $X$ is present in the background at all times.

Note that the annulus $A_1 \sim A_1'$ is naturally embedded in $X$. Let us now write the metric on this connected sum.

Set $z=(x,y)=(r\cos \phi, r\sin\phi\,\theta)$ with  $r\in\R^+$, $\phi\in[0,\pi/2]$, $\theta\in \S^{n-1}$ be the coordinates in $A_1$. We would like to glue the background metric $\bar g_{1,\delta}$ to the cylinder
$$\bar g_0:=\frac{dr^2+r^2d\phi^2+r^2\sin^2\phi d\theta^2}{r^2}=ds^2+d\phi^2+\sin^2\phi\, d\theta^2,$$
where we have defined the variable $s=-\log r$. For this,
consider the approximate cylindrical metric
$$\tilde g_{1,\delta}=\left(\chi(r)+(1-\chi(r))\frac{1}{r^2}\right)\bar g_{1,\delta},$$
for $\chi$ is a cutoff satisfying $\chi\equiv 1$ if $r\geq 2$, $\chi\equiv 0$ if $r\leq 1$. In the smaller ball $B_{1}(p_1)$, this metric is isometric to a cylinder plus a small error term.  Repeat for $g_2$.

Now we glue both metrics $\tilde g_{1,\delta}$, $\tilde g_{2,\delta}$  by a cutoff. Let $S_\epsilon:=-\log \epsilon$ and consider, as above,
$$A_1=\{\epsilon<r<1\}=\{0<s<S_\epsilon\}.$$
Let $\tilde \chi(s)$ be a cutoff function on $A_1$, satisfying $\tilde \chi\equiv 1$ for $s\leq  S_\epsilon/2-1$,  $\tilde \chi\equiv 0$ for $s\geq  S_\epsilon/2 + 1$.

Now repeat with $A_1'$, and recall that $s'=-s+S_\epsilon$. The approximate solution metric $\bar g_\epsilon$ is constructed as
$$\bar g_\epsilon=\tilde \chi(s) \tilde g_{1,\delta}+\tilde\chi'(s')\tilde g_{2,\delta}.$$
On $M:=\partial X$ we consider the metric $g_\epsilon:=\bar g_\epsilon|_{M}$.
 Note that $\bar g_\epsilon=\bar g_{i,\delta}$ on $X_i \backslash B_{C\sqrt \epsilon}(p_i)$, $i=1,2$.

It will be convenient to define a new variable $\bar s=s-S_\epsilon/2$, and to write the neck $A_1\sim A_1'$ as $\mathcal T_\epsilon=\{-S_\epsilon/2<s<S_\epsilon/2\}$. We let also $\mathcal M_\epsilon:=\mathcal T_\epsilon|_{\phi=\frac{\pi}{2}}$.
Then, as $\epsilon\to 0$, this neck  converges to
$\mathcal T_0=\R\times \S^n_+$ with the metric
$$
\bar g_0=d\bar s^2+d\phi^2+\sin^2\phi\, d\theta^2=d\bar s^2+g_{\S^n},
$$
where $g_{\S^n}$ is the standard metric on $\S^n_+$. Hence $(\mathcal T_0,\bar g_0)$ is half a solid cylinder, with boundary given by a lower dimensional cylinder $\mathcal M_0=\{\phi=\pi/2\}=\mathbb R\times \mathbb S^{n-1}$ and metric $g_0=d\bar s^2+d\theta^2$.

In Lemma \ref{lemma:error} we will show that we indeed have constructed a good approximate solution.

\section{The model cylinder}

From now on, we will take $\bar s$ as variable, and drop the bar in the notation if there is no risk of confusion.

Here we calculate precisely the Fourier symbol of the operator $P_{g_0}$ in the model cylinder. As in the previous section, let $\mathcal T_0=\R\times \S^n_+$ be the half solid cylinder with the metric
$$
\bar g_0=ds^2+d\phi^2+\sin^2\phi\, d\theta^2,
$$
where $s\in\mathbb R$, $\phi\in[0,\pi/2]$ and $\theta\in \S^{n-1}$. Its boundary (or, more precisely, its conformal infinity) is the model cylinder $\mathcal M_0=\{\phi=\pi/2\}=\mathbb R\times \mathbb S^{n-1}$ with metric
$$g_0=ds^2+d\theta^2.$$

Consider the spherical harmonic decomposition for $\mathbb S^{n-1}$. With some abuse of notation, let $\mu_m=m(m+n-2)$, $m=0,1,2,\dots$ be the eigenvalues of $-\Delta_{\mathbb S^{n-1}}$, repeated according to multiplicity. Then any function on $\mathbb R\times \mathbb S^{n-1}$ may be decomposed as $u(s,\theta)=\sum_{m} u_m(s) E_m(\theta)$, where $\{E_m(\theta)\}_m$ is the corresponding basis of eigenfunctions.

Following the construction of \cite{DelaTorre-Gonzalez}, let us write the hyperbolic AdS metric as
\begin{equation*}
g_0^+=\frac{1}{1+R^2}dR^2+(1+R^2)ds^2+R^2d\theta^2,
\end{equation*}
and make the change of variables
\begin{equation*}
R=\tan\phi.
\end{equation*}
This yields
\begin{equation}\label{g0+}
g_0^+=\frac{1}{\cos^2\phi}\left[d\phi^2+ds^2+\sin \phi d\theta^2\right]=\frac{\bar g_0}{\cos^2 \phi}.
\end{equation}
The function $\varrho=\cos\phi$ is a defining function for $\mathcal M_0$, which is the conformal infinity for this metric.

The conformal fractional Laplacian is calculated from the scattering operator of the Einstein metric $g_0^+$. The operator $P_{g_0}$ is defined as the Dirichlet-to-Neumann operator for the extension problem \eqref{Dirichletfv} in the metric $\bar g_0$. As it is explained in \cite{CG} (see also \cite{GG}), $P_{g_0}$ is precisely the conformal version of the half-Laplacian coming from scattering theory in the metric $g_0^+$. Here note that \eqref{g0+} is not written in normal form, but for the half-Laplacian case this is not an issue since the defining function does not appear explicitly in \eqref{Dirichletfv}.

The calculation of the scattering operator  and the conformal fractional Laplacian in this setting can be found in \cite{DelaTorre-Gonzalez,DDGW} and, in particular:

\begin{proposition}[\cite{DelaTorre-Gonzalez}]\label{prop:symbol}
 Fix $\gamma\in (0,\tfrac{n}{2})$ and let $P^{(m)}_{\gamma}$ be the projection of the operator $P^{g_0}_\gamma$ over each eigenspace $\langle E_m\rangle$. Then
$$\widehat{P_\gamma^{(m)} (f_m)}=\Theta_m(\xi) \,\widehat{f_m}.$$
Here $\,\hat \cdot\,$ denotes the usual Fourier transform
\begin{equation*}
\hat f(\xi)=\frac{1}{\sqrt{2\pi}}\int_{\mathbb R} f(s)e^{-i\xi s}\,ds.
\end{equation*}
The Fourier symbol $\Theta_m$ is given by
\begin{equation*}
\Theta_m(\xi)=2^{2\gamma}\frac{\Big|\Gamma\Big(\frac{1}{2}+\frac{\gamma}{2}
+\frac{1}{2}\left(\frac{n}{2}+m-1\right)+\frac{\xi}{2}i\Big)\Big|^2}
{\Big|\Gamma\Big(\frac{1}{2}-\frac{\gamma}{2}+\frac{1}{2}\left(\frac{n}{2}+m-1\right)
+\frac{\xi}{2}i\Big)\Big|^2}.
\end{equation*}
\end{proposition}

In this paper we are only interested in the case $\gamma=1/2$, so for the Fourier symbol of $P_{g_0}:=P_{1/2}^{g_0}$ simplifies to
\begin{equation*}
\Theta_m(\xi)=2\frac{\Big|\Gamma\Big(\frac{1}{4}
+\frac{1}{2}\left(\frac{n}{2}+m\right)+\frac{\xi}{2}i\Big)\Big|^2}
{\Big|\Gamma\Big(-\frac{1}{4}+\frac{1}{2}\left(\frac{n}{2}+m\right)
+\frac{\xi}{2}i\Big)\Big|^2}.
\end{equation*}
In addition, one can check that
\begin{equation}\label{normalization-constant}
Q_{g_0}=P_{g_0}(1)=\Theta_0(0)=\frac{2\,\Gamma\Big(\frac{n+1}{4}\Big)^2}
{\Gamma\Big(\frac{n-1}{4}
\Big)^2}=:c.
\end{equation}
 Note that this constant, in the paper \cite{DelaTorre-Gonzalez}, is denoted by $c_{n,1/2}$.

\section{Linear study}

We endow $X=X_1\sharp X_2$ with the approximate metric $\bar g_\epsilon$ constructed in Section  \ref{section:approximate}. Let $\mathcal L_\epsilon$ be the linearized operator \eqref{linearized} around the approximate metric $g_\epsilon=\bar g_\epsilon|_M$. The main step in the proof Theorem \ref{main} is to show that $\mathcal L_\epsilon$ is invertible, provided  that it is invertible on each $X_i$, $i=1,2$. Furthermore, we need to have uniform bounds for the inverse with respect to $\epsilon$.

Before we study the operator $\mathcal L_\epsilon$, let us look at the model cylinder $g_0$, for which the linearized operator has the simple formula
 \begin{equation*}
\mathcal L_0 v=P_{g_0}v-\kappa v,
 \end{equation*}
 where we have defined the constant $\kappa:=\frac{n+1}{n-1}c>0$.
 Taking into account the conformal transformation law for the  conformal fractional Laplacian, if we write our metric as
\begin{equation*}
g_0=\rho^{-2}(d\rho^2+\rho^2d\theta^2)=\rho^{-2}g_{\R^n},
\end{equation*}
for $\rho=e^{-s}$ and make the change
\begin{equation}\label{shift}
\tilde v=\rho^{-\frac{n-1}{2}}v,
\end{equation}
then we can see that $\mathcal L_0$ is a conjugate of a fractional Laplacian operator on $\mathbb R^n$ with critical Hardy potential, this is,
\begin{equation}\label{conjugation}
\mathcal L_0 v=\rho^{\frac{n+1}{n-1}}\left[(-\Delta_{\mathbb R^n})^{1/2}\tilde v-\frac{\kappa}{\rho}\tilde v\right]=:\rho^{\frac{n+1}{n-1}}\tilde{\mathcal L_0}\tilde v.
\end{equation}
One should keep in mind this conjugation in the analysis below, specially the shift in powers from \eqref{shift}.

\subsection{Elliptic regularity in weighted spaces}

Consider the weight $\omega_\epsilon$ defined as a smooth version of the function defined by
\begin{equation*}
\tilde{\omega}_\epsilon=\left \{
\begin{split}
&\cosh(s)/\cosh(S_\epsilon), \,\,\,\mbox{on}\,\,\mathcal T_\epsilon,\\
&1, \,\,\,\mbox{on}\,\,X\backslash \mathcal T_\epsilon.
\end{split} \right .
\end{equation*}
We require that $\omega_\epsilon$ is smooth, and agrees with $\omega_\epsilon$ except on a small neighborhood of the neck, where it remains bounded between $\frac{1}{2}$ and $2$. The previous weight has been introduced by Mazzeo and Pacard in \cite[Section 4]{Mazzeo-Pacard}. Denote
$$\overline\omega_\epsilon={\omega_\epsilon}|_{\phi=\pi/2},$$
its restriction to the boundary.
We note for comparison that when $s$ is fixed, and $\epsilon\rightarrow 0$,
$$\overline\omega_\epsilon\sim 2\epsilon\cosh(s).
$$

Now introduce  the weighted H\"older spaces $\omega_\epsilon^\mu \mathcal C^{k,\alpha}$
endowed with their natural norm
$$
\|v\|_{\mu,k,\alpha}=\|(\overline\omega_\epsilon)^{-\mu}v\|_{\mathcal C^{k,\alpha}}.
$$
We have the following lemma.
\begin{lemma} For each $\epsilon>0$, the map
$$\mathcal L_\epsilon: \overline \omega_\epsilon^\mu \mathcal  C^{2,\alpha}(M)\rightarrow \overline \omega_\epsilon^{\mu} \mathcal  C^{1,\alpha}(M)$$
is bounded.
\end{lemma}
\begin{proof}
Far from the point $p_i$ in $M_i$, $\mathcal L_\epsilon$ is a lower order perturbation of $(-\Delta_{g_i})^{1/2}$, $i=1,2$. There the weights do not play any role and thus $\mathcal L_\epsilon$ is bounded in $\mathcal  C^{2,\alpha}$.

On the other hand, near the point $p_1$, for instance, after conjugation, $\mathcal L_\epsilon$ is also a lower order perturbation of $(-\Delta_g)^{1/2}$. Thus the mapping properties of $\mathcal L_\epsilon$ can be deduced from the ones of $(-\Delta_g)^{1/2}$, and the Lemma follows by standard pseudodifferential calculus in weighted spaces (see the references \cite{Mazzeo:edge} and \cite{Mazzeo-Vertman} on regularity and boundary regularity for edge operators).

\end{proof}

\subsection{A Liouville theorem for the model metric}

One of the main results in \cite{ACDFGW} implies that the behavior of a singular solution to $\mathcal L_0 v=0$ is given  in terms of the inditial roots of the problem. This yields a Liouville theorem for $\mathcal L_0$ (see Theorem \ref{liouville} below). For this, in the notation of Proposition \ref{prop:symbol} for the spherical harmonic decomposition, set
 \begin{equation}\label{operator-m}
\mathcal L_0^{(m)} v:=P_{1/2}^m v-\kappa v,\quad m=0,1,\ldots
 \end{equation}
to be the $m$-th projection of the operator $\mathcal L_0$.

First, the indicial roots for  $\mathcal L_0^{(m)}$ are given in the following Lemma from \cite{ACDFGW}, and we refer to this paper for a more detailed asymptotic behavior. The crucial point of the later proofs is that, even though there exists an infinite sequence of inditial roots for each mode $m=0,1,\ldots$, the behavior of \eqref{operator-m}  is governed by the first inditial root pair; its real part will be denoted by $\gamma_m^\pm$.

\begin{lemma}\label{indicial} Fix $m=0,1,\ldots$
Both at $t=+\infty$ and $t=-\infty$,  there exist two sequences of indicial roots for $\mathcal L_0^{(m)}$, with no accumulation points:
\begin{equation*}
\{\sigma_j^{(m)}\pm i\tau_j^{(m)}\}_{j=0}^\infty\quad \text{and} \quad\{-\sigma_j^{(m)}\pm i\tau_j^{(m)}\}_{j=0}^\infty.
\end{equation*}
We denote  $\gamma_m^\pm:=\pm\sigma_0^{(m)}$.
Then
\begin{itemize}
\item[\emph{a)}] For the mode $m=0$, $\gamma_0^\pm=\sigma_0^{(0)}=0$, so we have a pair of inditial roots that are complex conjugates with real part zero.
\item[\emph{b)}]   For the remaining modes $m\geq 1$, $\tau_0^{(m)}=0$, so the first inditial root $\gamma_m^\pm$ is real. Actually, for the mode $m=1$, $\gamma_1^-:=-\sigma_0^{(1)}=-1$.

\item[\emph{c)}] $\{\gamma_m^+\}$ is an increasing sequence in $m$.

\item[\emph{d)}] In addition, for all $j\geq 1, \ m\geq 0$, $\sigma_j^{(m)}>\frac{n-1}{2}$.



\end{itemize}
\end{lemma}

\begin{proof}
This follows from  Lemma 7.1 in \cite{ACDFGW} for the special case $\gamma=\frac{1}{2}$ and $p=\frac{n+2\gamma}{n-2\gamma}$. There the indicial roots are written for the operator $\tilde{\mathcal L}_0$ from \eqref{conjugation},
 but they only differ by a translation of $\frac{n-1}{2}$. Moreover, the indicial roots at $t=+\infty$ and at $t=-\infty$ are the same in our case.

Note, however, that in \cite{ACDFGW} the lemma is stated only for $p<\frac{n+2\gamma}{n-2\gamma}$, but one can easily check that it also holds for $p=\frac{n+2\gamma}{n-2\gamma}$.
\end{proof}

Now we adapt Theorem 6.1, Remark 6.2. and Proposition 6.3 in \cite{ACDFGW} on the behavior of the Green's function for $\mathcal L_0^{(m)}$ to our setting (note that it corresponds to the unstable case):

\begin{theorem}\label{thm:Green}
Consider the equation
\begin{equation}\label{equation-m}
\mathcal L_0^{(m)} v=h.
\end{equation}
Assume that the right hand side $h$ in \eqref{equation-m} satisfies
\begin{equation*}
h(s)=\begin{cases}
O(e^{-\delta s}) &\text{as } s\to+\infty,\\
O(e^{\delta_0 s}) &\text{as } s\to-\infty,
\end{cases}
\end{equation*}
for some real constants $\delta,\delta_0$. It holds:
\begin{itemize}
\item[\textit{i.}] If $\delta>0$ and $\delta_0\geq 0$, then
a particular solution of \eqref{equation-m} can be written as
\begin{equation*}\label{solucion1}
v_m(s)=\int_{\mathbb R} h(\tilde s) {\mathcal G}_m(s-\tilde s)\,d\tilde s,
\end{equation*}
where
\begin{equation*}\label{Green1}
\mathcal G_m(s)=d_0 \sin({\tau_0s})\chi_{(-\infty,0)}(s)+\sum_{j=1}^\infty d_j e^{-\sigma_j |s|} \cos(\tau_j |s|)
\end{equation*}
for some constants $d_j$, $j=0,1,\ldots$. Moreover, $\mathcal G_m$ is an even $\mathcal C^{\infty}$ function when $s\neq 0$ and
\begin{equation}\label{decay-v}
v_m(s)=O(e^{-\delta s}) \text{ as }  s\to +\infty, \quad v_m(s)=O(e^{\delta_0 s}) \text{ as }  s\to -\infty.
\end{equation}

\item[\textit{ii.}] Now assume only that $\delta+\delta_0\geq 0$. If $\sigma_J<\delta<\sigma_{J+1}$ (and thus $\delta_0>-\sigma_{J+1})$, then a particular solution is
    \begin{equation*}\label{solucion2}
v_m(s)=\int_{\mathbb R} h(\tilde s) {\tilde{\mathcal G}}_m(s-\tilde s)\,d\tilde s,
\end{equation*}
where
\begin{equation*}
\tilde{\mathcal G}_m(s)=\sum_{j=J+1}^\infty d_j e^{-\sigma_j |s|} \cos(\tau_j |s|).
\end{equation*}
Moreover, $\tilde{\mathcal G}_m$ is an even $\mathcal C^{\infty}$ function when $s\neq 0$ and the same conclusion as in \eqref{decay-v} holds.

\item[\textit{iii.}]
All solutions of the homogeneous problem ${\mathcal L}_0^{(0)} v=0$ are of the form
\begin{equation}\label{mode-0}
\begin{split}
v(s)&=\,C_0^+ \sin({\tau_0^{(0)}s})+C_0^- \cos(\tau_0^{(0)} s)\\
&+\,\sum_{j=1}^\infty C_j^- e^{-\sigma_j^{(0)} s}
\cos(\tau_j^{(0)} s)+\sum_{j=1}^\infty C_j^+ e^{+\sigma_j^{(0)} s}
\cos(\tau_j^{(0)} s)
\end{split}
\end{equation}
for some real constants $C_j^-,C_j^+$, $j=0,1,\ldots$, and all solutions to  ${\mathcal L}_0^{(m)} v=0$ for $m\geq 1$ are of the form
\begin{equation}\label{mode-m}
\begin{split}
v(s)=\,\sum_{j=0}^\infty C_j^- e^{-\sigma_j^{(m)} s}
\cos(\tau_j^{(m)} s)+\sum_{j=0}^\infty C_j^+ e^{+\sigma_j^{(m)} s}
\cos(\tau_j^{(m)} s)
\end{split}
\end{equation}

\item[\textit{iv.}] The only solution to \eqref{equation-m}, in both the cases \emph{i.} and \emph{ii.}, with decay as in \eqref{decay-v} is precisely $v_m$. Moreover we have unique continuation: any solution that decays faster than any exponential at $t\to \infty$ or $t\to -\infty$ must vanish identically.
\end{itemize}

\end{theorem}

Then one obtains, as a consequence of parts \emph{iii.}-\emph{iv.} above, the following Liouville theorem by choosing the weight $\mu$ accordingly:

\begin{theorem}\label{liouville} Assume that $\mu\in (-\frac{n-1}{2},0)$ and $\mu \neq -\sigma_0^{(m)}$ for $m\geq 0$. Then any solution to  $\mathcal L_0v=0$ satisfying $|v(s,\theta)|\leq ce^{\mu s}$ or $|v(s,\theta)|\leq ce^{-\mu s}$
must vanish everywhere.
\end{theorem}

\begin{proof}

From statements \emph{iii.} of Theorem \ref{thm:Green}, all the solutions to the homogeneous problem $\mathcal L^{(m)}_0v=0$ are given by \eqref{mode-0} and \eqref{mode-m}. If the weight $\mu\in (-\frac{n-1}{2},0)$, then by \emph{d)} of Lemma \ref{indicial}, $e^{\sigma_j^{(m)}s}>e^{\mu s}$ for $s>0$ and $e^{-\sigma_j^{(m)}s}>e^{\mu s}$ for $s<0$ when $j\geq 1$. So all $C_j^{\pm}=0$ for $j\geq 1, \ m\geq 0$. Moreover, if $\mu\neq -\sigma_0^{(m)}$ for all $m$, using \emph{c.} of Lemma \ref{indicial}, there exists $J\geq 0$ such that $-\sigma_{0}^{J+1}<\mu<\sigma_0^{J}$. Similar to the previous argument again, $e^{\mu s}<e^{-\sigma_0^{i}s}$ when $s>0$ for $i\leq J$ (here we agree that $-\sigma_0^{(i)}=\sigma_0^{(-i)}$ if $i<0$) and $e^{\mu s}<e^{-\sigma_0^{i}s}$ when $s<0$ for $i\geq J+1$, so $C_0^{\pm}=0$ for all $m$. Combining all the above results, one can elliminate all the nontrivial solutions to $\mathcal L_0v=0$. Then in this weighted space, only the trivial solution exists.
\end{proof}


\subsection{Removability of singularities}

We present here a standard result concerning removability of isolated singularities for an elliptic problem if the solution grows slower than its Green's function.

\begin{proposition}\label{noyaupointe}
Let $(X,\bar g)$ be a compact manifold with boundary $(M,g:=\bar g|_M)$.
Assume $\mathcal L_g v=0$ on $M\backslash\{p\}$ and
\bel{ud}
|v(\cdot)|<C\dist(p,\cdot)^{\mu},\;\;\;\mu\in(1-n,0)
\ee
then $u$ extends on all of $M$ so that $\mathcal L_g v=0$ on $M$.
\end{proposition}

\begin{proof}
This is a standard result but we have not been able to find an exact reference for it.
In normal coordinates near $p$, $\mathcal L_g$ is a lower order perturbation of $(-\Delta_{g})^{1/2}$ plus a zero-th order term. A classical B\^ocher theorem for the half-Laplacian on a punctured ball (see, for instance, \cite{Li-Wu-Xu}) yields that $v$ should be the sum of a smooth  function plus a constant times the fundamental solution of the half-Laplacian. But our hypothesis \eqref{ud} prevents this second case and thus the constant must be zero. Then $v$ can be extended smoothly all the way across the singularity.
\end{proof}

\begin{remark}
This proposition has to be compared to the last statement in Theorem $1.1$ in \cite{GMS}.
\end{remark}

\subsection{Non-degeneracy of the linearized operator}

We now prove the main result of this section.

\begin{proposition}\label{invL}
In the setting of Theorem \ref{main}, assume that
the maps $\mathcal L_{g_i}: \mathcal C^{1,\alpha}(M_i)\rightarrow \mathcal C^{\alpha}(M_i)$ on $M_i$, $i=1,2$, have trivial kernel.
Then for
\begin{equation}\label{choice-mu}
\mu\in\Big(-\frac{n-1}{2},0\Big)\quad\text{and}\quad \mu \neq -\sigma_0^{(m)} \text{ for }m\geq 0,
\end{equation}
the linearized operator $\mathcal L_\epsilon$ is invertible on $M=\partial X$.
Moreover, its inverse can be bounded independently of $\epsilon$ (small enough).
\end{proposition}

\begin{proof}
We only have to show that there exist a constant $C$ such that for all $v\in  \overline \omega_\epsilon^\mu \mathcal C^{2,\alpha}$,
$$
\|v\|_{\mu,1,\alpha}\leq C\|\mathcal L_\epsilon v\|_{\mu,0,\alpha}
$$
for some $C$ independent of $\epsilon$ small enough. We follow the proof of proposition 11 of \cite{Mazzeo-Pacard}.
By contradiction, there would exist  sequences $\epsilon_j\rightarrow 0$ and $\{v_j\}$ for which
$$
\|v_j\|_{\mu,1,\alpha}=1,
$$
while
$$
\|\mathcal  L_{\epsilon_j} v_j\|_{\mu,0,\alpha}\rightarrow 0.
$$
In particular this gives the estimates
$$
|v_j(y)|\leq \overline \omega_{\epsilon_j}^\mu(y),
$$
and
$$
|\mathcal L_{\epsilon_j} v_j(y)|\leq \eta_j \overline \omega_{\epsilon_j}^{\mu}(y)
$$
for all $y\in M$, where $\eta_j\rightarrow 0$.
Let $q_j$ a point where  $|v_j| (\overline \omega_{\epsilon_j})^{-\mu} $ attains
its maximum. Possibly passing to a subsequence, different  cases can occur:
\begin{itemize}
\item[\emph{(i)}] $\{q_j\}$ converges to a point $q$ in $M_1\backslash \{p_1\}$ or $M_2\backslash \{p_2\}$,\\
\item [\emph{(ii)}] $\{q_j\}$ lies in $\mathcal M_\epsilon$,
and $|s_j|\leq C$.\\
\item [\emph{(iii)}] $\{q_j\}$ lies in $\mathcal M_\epsilon$,
and $|s_j|\rightarrow\infty$.\\

\end{itemize}
We start by ruling out the case $(i)$. We may assume that $q_j\rightarrow q\in M_1\backslash \{p_1\}$.
Then $\{v_j\}$ converges in $\mathcal C^{2,\alpha}$ on any compact set of $M_1\backslash \{p_1\}$ to $v$ such that
$\mathcal L_{g_1} v=0$. Now recall that on
$B_C(p_1)\backslash B_\epsilon(p_1)$ we have $\overline \omega_\epsilon$ varying from $\frac2{\epsilon+\epsilon^{-1}}\sim2\epsilon=2d(p_1,\cdot)$ to $1$, so using that  $|v_j(y)|\leq \overline \omega_{\epsilon_j}(y)^\mu$, we obtain as $\epsilon\to 0$,  that $|v|\leq C d(p_1,\cdot)^{\mu}$. Hence by Proposition \ref{noyaupointe} and the hypothesis that $\mathcal L_{g_1}$ has trivial kernel, we have  $v\equiv 0$. However, $v(q)\neq 0$ since $\|v\|_{\mu,1,\alpha}=1$. This yields a contradiction.

We now study cases $(ii)$ and $(iii)$. For the case $(ii)$, as $\epsilon\rightarrow 0$, ${\mathcal M}_\epsilon\rightarrow \mathcal M_0$,
and by the assumption that $\{q_j\}$ stays in a compact set of ${\mathcal M}_0$, we may take $q_j\rightarrow q\in{\mathcal M}_0$.
But when $y\in\mathcal M_0$,
$$
\frac{\overline\omega_{\epsilon_j}(y)}{\overline\omega_{\epsilon_j}(q_j)}=\frac{\cosh s}{\cosh s_j}\rightarrow c'\cosh s,
$$
where $c'>0$ and the convergence is $\mathcal C^\infty$ on any compact set.
We define $$u_j:=\overline\omega_{\epsilon_j}(q_j)^{-\mu}v_j,$$ so $u_j\rightarrow u$ with $u(q)\neq 0$, $\mathcal L_0 u=0$ (where the last operator is with respect to the metric $g_0$)
and $|u(y)|\leq c \cosh^\mu(s)\leq ce^{\mu |s|}$.  By Theorem \ref{liouville},  one has $u\equiv 0$, which contradicts to $u(q)\neq 0$.

Now consider the case $(iii)$ and define
$$
u_j(s,\theta):=\overline\omega_{\epsilon_j}(q_j)^{-\mu}v_j(s+s_j,\theta).
$$
As $|v_j(s+s_j,\theta)|\leq c \overline \omega_{\epsilon_j}(s+s_j,\theta)^{\mu}$ and
$$
\frac{\overline \omega_{\epsilon_j}(s+s_j)}{\overline \omega_{\epsilon_j}(s_j)}=
\frac{\cosh(s+s_j)}{\cosh S_{\epsilon_j}}\frac{\cosh S_{\epsilon_j}}{\cosh s_j}\rightarrow
\left\{\begin{array}{l}e^{s},\mbox{ if } s_j\rightarrow+\infty,\\
e^{-s},\mbox{ if } s_j\rightarrow-\infty,
\end{array}\right.
$$
 then $u_j\rightarrow u$ with $u(0,\theta)\neq 0$ and $u$ is a solution to $\mathcal L_{0} u=0$. Since
 $|u(y)|\leq c (e^s)^\mu$ everywhere or $|u(y)|\leq c (e^{-s})^\mu$ everywhere,  again by Theorem \ref{liouville}, we arrive at a contradiction.
\end{proof}

\begin{remark}
In Proposition \ref{invL} it is enough to assume that $M_1,M_2$ are non-degenerate (in the sense that $\mathcal L_\epsilon$ is well behaved only after adding a \emph{deficiency} space). This is the case if some $M_i$, $i=1,2$, is a Delaunay metric, for instance. However, this generalization only adds technicalities in the analysis and we are not interested in those here.
\end{remark}

\subsection{Surjectivity}\label{subsection:surjectivity}

To find an inverse for the linearized operator $\mathcal L_\epsilon$ we set up the problem in the Hilbert space $L^2(X)$. For this functional analysis part we do not need to introduce weights since the manifold $X=X_1\sharp X_2$ is compact; in the noncompact case, one would need to consider  weighted Lebesgue spaces and to work with deficiency spaces (see \cite{MPU} and later references).

First note that  $\overline \omega_\epsilon^\mu\mathcal C^{\alpha}\hookrightarrow L^2$ if $\mu<0$. We claim that, in weighted Lebesgue spaces, $\mathcal L_\epsilon:L^2(M)\to L^2(M)$ is injective. To see this, let $v\in L^2$ such that $\mathcal L_\epsilon v=0$. By Proposition \ref{invL}, it is enough to show that $v\in \overline\omega_\epsilon^\mu\mathcal C^{1,\alpha}$. But this is a consequence of Theorem \ref{thm:Green} and elliptic regularity again.

Next, since $\mathcal L_\epsilon$ is a self-adjoint operator in $L^2$, surjectivity follows immediately. By elliptic regularity,one can obviously produce a right inverse $G_\epsilon$ in the weighted space $\overline \omega_\epsilon^\mu\mathcal C^\alpha$. Finally, Proposition \ref{invL}  shows that  this inverse has norm uniformly bounded in $\epsilon$.

\section{The construction: proof of Theorem \ref{main}}\label{construction}

Let $g_\epsilon$ be the metric constructed in Section \ref{section:approximate} for the connected sum $M=\partial X$, $X=X_1\sharp X_2$.
The proof of Theorem \ref{main} follows from solving
\begin{equation*}
\mathcal Q_{g_\epsilon}(f):=f^{-\frac{n+1}{n-1}}P_{g_\epsilon}f=c \,\,\text{ constant.}
\end{equation*}
Recalling \eqref{problem}, this equation should always be coupled with
\begin{equation}\label{equation-extension}
L_{\bar g_\epsilon}(f)=0 \quad \text{in }X,
\end{equation}
even if not always explicitly written.
Without loss of generality, let us normalize the constant  $c$  as in \eqref{normalization-constant}.

The procedure is by now standard, using a fixed point argument (see, for instance, Section 9 in  \cite{Mazzeo-Pacard:construction}). Define
$f=1+v$.
Expanding in Taylor series the operator $\mathcal Q$, one gets
$$\mathcal Q_{g_\epsilon} (1+v)=\mathcal Q_{g_\epsilon}(1)-c+\mathcal L_\epsilon v + \mathcal N_\epsilon(v),$$
where $\mathcal N_\epsilon(v)$ is quadratic in $v$. Denote
$$\mathcal A_\epsilon v:= -G_\epsilon (\mathcal Q_{g_\epsilon}(1) -c+\mathcal N_\epsilon(v)),$$
where $G_\epsilon:\overline w_\epsilon^{\mu}\mathcal C^{0,\alpha}\to \overline w_\epsilon^{\mu}\mathcal C^{1,\alpha}$ is the right inverse constructed in the previous section, for $\mu$ satisfying \eqref{choice-mu}.

Then equation \eqref{theequation} is reduced to $v=\mathcal A_\epsilon v$, and we just need to show that $\mathcal A_\epsilon$ is contractive for $\epsilon$ small enough. \\

First we check that $g_\epsilon$ is indeed a good approximate metric:

\begin{lemma}\label{lemma:error} Assume that $\mu<0$. Then
\begin{equation*}
\|Q_{g_\epsilon}-c\|_{\mu,0,\alpha}\leq C_\epsilon,
\end{equation*}
where $C_\epsilon\rightarrow0$ as $\epsilon\to 0$.
\end{lemma}

\begin{proof}
Again, we will prove the result for the weighted $L^\infty$ norm $\|\cdot\|_{\mu,0,0}$; to pass to the weighted H\"older norm is standard.

In order to calculate $Q_{g_\epsilon}$, we set up the problem \eqref{problem} with Dirichlet condition identically one, this is
\begin{equation}\label{eq}
L_{\bar g_\epsilon} u=0\quad\text{on }X.
\end{equation}
Our second remark is that we are dealing with edge operators on manifolds with boundary in the sense in the sense of \cite{Mazzeo:edge,Mazzeo-Vertman}, so their mapping properties are well known.

Let us relate the metric $\bar g_\epsilon$ to $\bar g_0$. Assume first that we are in a neighborhood where $\tilde g_{1,\delta}$ lives, this is, $X_1\backslash B_1(p_1)$.  If we write $\tilde g_{1,\delta}=\varphi^{\frac{4}{n-1}}\bar g_1$, recalling the conformal transformation law for the conformal Laplacian \eqref{transformation-law} and the fact that $\bar g_1$ is scalar flat, then equation \eqref{eq} is simply
\begin{equation*}
-\Delta_{\bar g_1}(\varphi u)=0.
\end{equation*}
But $\varphi=1$ outside a neighborhood $B_2(p_1)$, so we recover the original problem for $\bar g_1$.

Inside the neighborhood $B_2(p_1)$ we use normal coordinates $\bar g_2=|dz|^2(1+o(1))$, so up to a small error, we can approximate \eqref{eq} by
\begin{equation}\label{eq2}
-\Delta_{|dz|^2}(\varphi u)=0.
\end{equation}
On the other hand, in $B_1(p_1)$ but before gluing to $\bar g_2$, $\varphi=r^{-\frac{{n-1}}{2}}$, so  \eqref{eq2} can be written as a problem in the cylinder, with metric $\bar g_0$, and it can be included in the theory of edge operators, which are well behaved with respect to weights. In the intermediate neighborhood $B_2(p_1)\backslash B_1(p_1)$ we have error terms depending on $\nabla \varphi$ and $\Delta \varphi$, but these are of lower order.

Let us look closer to the neighborhood $B_1(p_1)$. In this region, equation \eqref{eq} is written as
\begin{equation*}
-\Delta_{\bar g_0} u+\tfrac{n-1}{4n}R_{\bar g_0}u=0.
\end{equation*}
The gluing $\tilde g_{1,\delta}$ to $\tilde g_{2,\delta}$ using the cutoff  $\tilde\chi$  happens in a region $\{C^{-1}\epsilon^{1/2}\leq r\leq C \epsilon^{1/2}\}$. Since $\Delta(\tilde\chi u)-\tilde\chi\Delta u=2\nabla \tilde\chi\cdot\nabla u+ u\Delta \tilde\chi =:\tilde h$, we are creating two new error terms (depending on $\epsilon$). Let us translate everything in the notation  $s$ (for which our weights are adapted). Taking into account the conformal transformation law for $\mathcal L$, our  equation is simply
\begin{equation*}
\mathcal L_{0} u=h
\end{equation*}
for $h:=\tilde h r^{-\frac{n+1}{2}}$. This $f$ has support only on $\Sigma:=\{S_\epsilon/2-1\leq s\leq S_\epsilon/2+1\}$.

Adding the Dirichlet condition $u\equiv 1$ to the extension problem \eqref{eq}, we recover $Q_{g_\epsilon}$ as the Dirichlet-to-Neumann operator
\begin{equation}\label{eq4}
Q_{g_\epsilon}=-\partial_\nu u+\tfrac{n-1}{2}H_{g_\epsilon}u|_{M}.
\end{equation}
When we couple the boundary operator with \eqref{eq}, we are in the framework of studying edge operators on manifolds with boundary from \cite{Mazzeo-Vertman}.

We will compare \eqref{eq}-\eqref{eq4} to the model (zero scalar curvature in the bulk, constant mean curvature of the boundary). Standard regularity estimates yield that the $H^2$ norm of $u$ is bounded, for instance, in terms of the $L^2$ norm in the right hand side. It is standard to add a suitable weight to these estimates. So, as we are comparing two linear problems,  to have good estimates for the Neumann boundary term \eqref{eq4} it is enough to estimate, in a weighted $L^2(dt)$ norm, the error terms produced in the gluing process above, for the weight $\omega_\epsilon^\mu$. Since the weight has a factor $\epsilon^{-\mu}$ for $\mu<0$, we only need to take into account the neck, which is the noncompact part. But here the error term only lives in the neighborhood $\Sigma$, so
\begin{equation*}
\int_\Sigma (\omega_\epsilon)^{-2\mu} h^2\, dt d\phi d\theta \sim \left(S_\epsilon^n \epsilon^{-\mu}\right)^{1/2}\to 0 \quad\mbox{as}\quad \epsilon\to 0.
\end{equation*}
This completes the proof of the Lemma.\\
\end{proof}

We finally check that $\mathcal A_\epsilon$ is a contraction in a suitable space. Define
$$\mathcal B_{\mu} (r_\epsilon)=\left \{  v\in  (\overline\omega_\epsilon)^{\mu} \mathcal C^{1,\alpha}:\,\, \| v\|_{1,\alpha,\mu} \leq r_\epsilon \right \}.$$

\begin{lemma}
For sufficiently small $\epsilon >0$, there exists a radius $r_\epsilon$ such that
$$\mathcal A_\epsilon (\mathcal B_\mu (r_\epsilon)) \subset \mathcal B_\mu (r_\epsilon)$$
and
\begin{equation*}
\|\mathcal A_\epsilon v_1-\mathcal A_\epsilon v_2\|_{1,\alpha,\mu}\leq \frac{1}{2}\|v_1-v_2\|_{1,\alpha,\mu}
\end{equation*}
for all $v_1,v_2\in \mathcal B_\mu(r_\epsilon)$.
\end{lemma}

\begin{proof}
For simplicity, we will show the result for the weighted $\mathcal C^0$ norm, and not $\mathcal C^\alpha$. Modifications in the general case are straightforward.

Recalling the definition of the operator from \eqref{theequation}, we have
$\mathcal Q_{g_\epsilon}(1)=Q_{g_\epsilon}$, so Lemma \ref{lemma:error} yields an estimate for the  term $\mathcal Q_{g_\epsilon} (1)-c$. Then, taking into account the discussion in Section \ref{subsection:surjectivity}, to complete the proof of the Lemma we just need to bound
\begin{equation*}
\|\mathcal N_\epsilon(v_1)-\mathcal N_\epsilon(v_2)\|_{0,0,\mu}\leq o(1) \|v_1-v_2\|_{1,0,\mu}.
\end{equation*}
The proof of this inequality is standard and therefore we omit the argument.
\end{proof}

From the previous lemma, we have constructed a solution $v\in (\overline\omega_\epsilon)^\mu$ such that the metric $g'=(1+v)^{\frac{4}{n-1}}g_\epsilon$ has constant $Q$-curvature. The usual maximum principle for Dirichlet-to-Neumann operators (see \cite{Cabre-Sire:I}, for instance) implies that $1+v$ is a positive function. Moreover, by equation \eqref{equation-extension}, the metric $\bar g'=(1+v)^{\frac{4}{n-1}}\bar g_\epsilon$, has zero scalar curvature. This completes the proof of Theorem \ref{main}.

\qed

\bigskip

\noindent\textbf{Acknowledgements.}  Part of the work was done when W. Ao was visiting the math department of Universidad Aut\'onoma de Madrid in June 2018, she thanks the hospitality of the department.  M.d.M. Gonz\'alez is supported by the Spanish government grants MTM2014-52402-C3-1-P and MTM2017-85757-P. Y.S. would like to thank E. Delay and F. Pacard for useful dicussions. \\


%

\bibliographystyle{plain}

\end{document}